\documentclass[11pt,a4paper]{amsart}
\usepackage[foot]{amsaddr}
\usepackage{texmac,enumitem}

\operators{Stab,Prim,Inn,Out,ss,B,G,R,K,Inf,Def,N,rank}

\bbsymbols{b}{C,F,Z,Q}
\calsymbols{c}{H,U,X,Y,P,M,D,S,T,K}
\let\le\leqslant
\let\ge\geqslant
\def\qe#1{$#1$-quasi-elementary}
\def\dsum_#1_#2{\sum_{{#1}\atop {#2}}}
\def\rar{\rightarrow}
\def\ck{\cK}

\begin{document}

\title[Brauer relations in positive characteristic]
{Relations between permutation representations in positive characteristic}
\author{Alex Bartel$^1$}
\address{$^1$School of Mathematics and Statistics, University of Glasgow, University Place,
Glasgow G12 8SQ, United Kingdom}
\author{Matthew Spencer$^2$}
\address{$^2$Mathematics Institute, Zeeman Building, University of Warwick,
Coventry CV4 7AL, UK}
\email{alex.bartel@glasgow.ac.uk, matthew.james.spencer91@gmail.com}
\llap{.\hskip 10cm} \vskip -0.8cm
\maketitle

\begin{abstract}
Given a finite group $G$ and a field $F$, a $G$-set $X$ gives rise to an
$F[G]$-permutation module $F[X]$. This defines a map from the Burnside ring
of $G$ to its representation ring over $F$. It is an old problem in
representation theory, with wide-ranging applications in algebra, number theory,
and geometry, to give explicit generators of the kernel $\K_F(G)$ of this map,
i.e. to classify pairs of $G$-sets $X$, $Y$ such that $F[X]\cong F[Y]$. When
$F$ has characteristic $0$, a complete description of $\K_F(G)$ is now known. In
this paper, we give a similar description of $\K_F(G)$ when $F$ is a field of
characteristic $p>0$ in all but the most complicated case, which is when $G$
has a subquotient that is a non-$p$-hypo-elementary $(p,p)$-Dress group.
\end{abstract}



\section{Introduction}\label{s:intro}
In the present paper, we study which finite $G$-sets $X$, $Y$, for a finite group
$G$, give rise to isomorphic linear permutation representations over a field of
positive characteristic. To explain the precise problem and the main result,
we need to recall some terminology.

Let $F$ be a commutative ring, and $G$ a finite group. The \emph{Burnside ring}
$\B(G)$ of $G$ has one generator $[X]$ for every finite $G$-set $X$,
and relations $[X]+[Y]=[Z]$ for all isomorphisms $X\sqcup Y \cong Z$ of $G$-sets,
with multiplication being defined by $[X]\cdot [Y]=[X\times Y]$. Since every
finite $G$-set is a finite disjoint union of transitive $G$-sets, and
every transitive $G$-set is isomorphic to a set of the form $G/H$, where $H$
is a subgroup of $G$, with $G/H$ isomorphic to $G/H'$ if and only if $H$ is
$G$-conjugate to $H'$, we deduce that as a group $\B(G)$ is free abelian on
the set of conjugacy classes of subgroups of $G$. We will therefore
write elements $\Theta$ of $\B(G)$ as linear combinations of subgroups of $G$,
which are always understood to be taken up to conjugacy. We will also sometimes
refer to these (representatives of) conjugacy classes of subgroups as \emph{the terms
of $\Theta$}, so that if $\Theta\in \B(G)$ and $H$ is a subgroup of $G$, we may talk
about \emph{the coefficient of $H$ in $\Theta$}. The \emph{representation ring} $\R_F(G)$
of $G$ over $F$ has a generator $[M]$ for every finitely generated
$F[G]$-module $M$, and relations $[M]+[N]=[O]$ for all isomorphisms $M\oplus N\cong O$
of $F[G]$-modules,
with multiplication being defined by $[M]\cdot [N]=[M\otimes_F N]$, where $G$ acts
diagonally on the tensor product. This is not to be confused with the ring of
Brauer characters of $F[G]$-modules, which is also often denoted by $\R_F(G)$, but which
will not feature in our paper.

There is a natural map $\B(G)\rar \R_F(G)$, sending the class of a $G$-set $X$
to the class of the associated permutation representation $F[X]$.
Let $\K_F(G)$ denote the kernel of this map. Its elements will be referred to as
\emph{Brauer relations over $F$}, or, once the choice of $F$ is understood, just as \emph{relations}. It is
easy to see that if $F$ is a field, then
the structure of $\K_F(G)$ only depends on $G$ and on the characteristic of $F$.
A good understanding of Brauer relations over fields of different characteristics
has many applications in number theory and geometry.
Brauer and Kuroda were, independently, the first to systematically investigate
this phenomenon when they used the non-triviality of $\K_{\Q}(G)$ to derive
interesting relations between class groups of number fields \cite{Bra,Kur}.
Since then, Brauer relations have been found to give rise to many interesting
relations between different invariants of number fields
\cite{Wal,Smi-01,BB-04,Bar-11}, of elliptic and modular curves
\cite{squarity,tamroot,SmiEdix}, and of Riemannian manifolds \cite{Sunada,GSdS,BP}.
In these applications, $F$ is usually taken to be a field, and one obtains
interesting information already by analysing $\K_{\Q}(G)$, but the sharpest results are
typically achieved if one knows precisely for what primes $p$ a given element of
$\B(G)$ is a relation over a field of characteristic $p$.

When $F$ is a field of characteristic $0$, a set of explicit generators of $\K_F(G)$
for all $G$ has been determined by the first author and T. Dokchitser
in \cite{bra1}, following important advances
due to Brauer himself \cite{Bra}, Langlands \cite{Lan-70}, Deligne \cite{Del-73},
Snaith \cite{Sna-88}, Tornehave \cite{Tor-84}, and Bouc \cite{Bouc}.
In contrast, almost nothing seems to be known about explicit generators of
$\K_F(G)$ when $F$ is a field of positive characteristic. The main result of
the present paper, Theorem \ref{thm:main} below, addresses that situation
by making substantial progress towards a complete classification.

The standard approach to problems of this kind is to view an element of
$\K_F(G)$ as ``uninteresting'' if it comes from a proper subquotient of $G$ (see
\S\ref{sec:first}). Call such a relation \emph{imprimitive}, and let $\Prim_F(G)$
denote the quotient of $\K_F(G)$ by the subgroup generated by the imprimitive
relations. If one can find a set of generators for $\Prim_F(G)$ for each finite
group $G$, then one can give a complete description of $\K_F(G)$: for every
finite group $G$, every element of $\K_F(G)$ is a linear combination of elements
of the form $\Ind\Inf\Theta$ as $\Theta$ runs over generators of $\Prim_F(U)$
for all subquotients $U$ of $G$. Such a description turns out to be
ideally suited for the applications in number theory and geometry mentioned above. 

If $p$ and $q$ are prime numbers, then a finite group is called \emph{\qe{q}} if
it has a normal cyclic subgroup of $q$-power index, equivalently if it is a split
extension of a $q$-group by a cyclic group of order coprime to $q$;
a finite group is called \emph{$p$-hypo-elementary} if it has a normal
$p$-subgroup with cyclic quotient, equivalently if it is a split
extension of a cyclic group of order coprime to $p$ by a $p$-group;
a group is called a \emph{$(p,q)$-Dress group} if it has a normal $p$-subgroup
with \qe{q} quotient.
A finite group
is called quasi-elementary if it is \qe{q} for some prime number $q$.
The main result of the present paper is the following.

\begin{theorem}\label{thm:main}
Let $F$ be either a field of characteristic $p>0$, or a discrete valuation ring
with finite residue field of characteristic $p$. Let $G$ be a finite group, and
suppose that $\Prim_F(G)$ is non-trivial. Then:
\begin{enumerate}[leftmargin=*, label={\upshape(\Alph*)},itemindent=.5cm,labelwidth=\itemindent,align=left,labelsep=0.2cm]
\item\label{item:structG} the group $G$ is not $p$-hypo-elementary,
and in addition $G$ satisfies one of the following conditions:
\begin{enumerate}[leftmargin=*, label={\upshape(\roman*)},itemindent=.5cm,labelwidth=\itemindent,align=left,labelsep=0.2cm]
\item\label{item:main_qquasielem} the group $G=C\rtimes Q$ is quasi-elementary
of order coprime to $p$, where $C$ is cyclic and $Q$ is a $q$-group for some prime number $q$, and
either $C$ is not of prime order, or $Q$ does not act faithfully on $C$;
\item\label{item:main_Serre} there are a normal elementary abelian $l$-subgroup
$W\cong (C_l)^d$ of $G$, where $l$ is a prime number and $d\ge 1$ is an integer,
and a $(p,q)$-Dress subgroup $D$ of $G$, where $q$ is a prime number, such that
$G=W\rtimes D$, with $D$ acting faithfully on $W$; moreover, either $D$ acts
irreducibly on $W$, or $G=(C_l\rtimes D_1)\times (C_l\rtimes D_2)$, where
$D_1$, $D_2$ are cyclic $q$-groups;
\item\label{item:main_nonabelianSerre} there is an exact sequence
$$
1 \to S^d \to G \to D\to 1,
$$
where $S$ is a non-abelian simple group, $d\ge 1$ is an integer,
and $D$ is a $(p,q)$-Dress group for a prime number $q$, such that
the natural map $D\to \Out(S^d)$ is injective, and $S^d$ is
a minimal non-trivial normal subgroup;
\item the group $G$ is a $(p,p)$-Dress group.
\end{enumerate}
Moreover,
\item\label{item:generatorsPrim}
in the cases \ref{item:main_qquasielem} -- \ref{item:main_nonabelianSerre}, the
structure of $\Prim_F(G)$ and a set of generators are as follows.
\begin{enumerate}[leftmargin=*, label={\upshape(\roman*)},itemindent=.5cm,labelwidth=\itemindent,align=left,labelsep=0.2cm]
\item One has $\Prim_F(G)=\Prim_{\Q}(G)$, and the latter is described by
\cite[Theorem A, case (4)]{bra1}.
\item If $D$ is $p$-hypo-elementary, then $\Prim_F(G)\cong \Z$, and otherwise
$\Prim_F(G)\cong\Z/q\Z$.
In both cases, $\Prim_F(G)$ is generated by $\Theta$ defined as follows:
\begin{enumerate}[leftmargin=*,label={\upshape(\alph*)},itemindent=.5cm,labelwidth=\itemindent,align=left,labelsep=0.2cm]
\item\label{item:mn} if $d=1$ and $D=C_{mn}=C_m\times C_n$ is cyclic of order $mn$, where $m$, $n>1$
are coprime integers,
then $\Theta=G-C_{mn}+\alpha(C_n-C_l\rtimes C_n)+\beta(C_m-C_l\rtimes C_m)$,
where $\alpha$, $\beta$ are any integers satisfying $\alpha m+\beta n = 1$;
\item\label{item:q^k} if $d=1$ and $D=C_{q^{k+1}}$ is cyclic of order $q^{k+1}$,
where $k\in \Z_{\ge 0}$, then $\Theta=C_{q^k}-qC_{q^{k+1}}-C_l\rtimes C_{q^k}+qG$;
\item\label{item:d>1} if $d\ge 2$, then $\Theta = G-D+\sum_{U\le_G W\atop (W:U)=l} (U\N_D(U)-W\N_D(U))$,
where the sum runs over a full set of $G$-conjugacy class representatives of index
$l$ subgroups of $W$, and where $\N_D(U)$ denotes the normaliser of $U$ in $D$.
\end{enumerate}
\item If $D$ is $p$-hypo-elementary, then $\Prim_F(G)\cong \Z$, and otherwise
$\Prim_F(G)\cong\Z/q\Z$.
In both cases, $\Prim_F(G)$ is generated by any relation of the form $G+\sum_{H\lneq G}a_H H$,
where $a_H$ are integers.
\end{enumerate}
\end{enumerate}
\end{theorem}
Explicit formulae for relations as in Theorem
\ref{thm:main}\ref{item:generatorsPrim}\ref{item:main_nonabelianSerre} can be
derived from \cite{Boltje} or from \cite{BK}.

Let us briefly sketch the main ingredients of the proof and the structure of the paper.
In section \ref{sec:first}, we recall the basic formalism of Brauer relations
and results from the literature that we will need in the rest of the paper.
The most important one of these is Theorem \ref{thm:quotients}, which places
tight restrictions on the possible quotients of a finite group $G$ for which
$\Prim_F(G)$ is non-trivial. For example it states that if $G$ is a finite group
for which $\Prim_F(G)$ is non-trivial, then there exists a prime number $q$
such that all proper quotients of $G$ are $(p,q)$-Dress groups.
If moreover $G$ itself is not a $(p,q')$-Dress group for any prime number
$q'$, then $\Prim_F(G)$ is cyclic, and is generated by any relation
of the form $\Theta=G+\sum_{H\lneq G}a_HH$, where $a_H\in \Z$. This almost
immediately implies the conclusions of the theorem when $G$ is not soluble 
--  see part \ref{item:main_nonabelianSerre} of the conclusion.

In Section \ref{sec:soluble}, we turn our attention to soluble groups.
First, we prove in Theorem \ref{thm:pqDress} that if $q$ is a prime number
different from $p$, and $G$ is a $(p,q)$-Dress group with a non-trivial
normal $p$-subgroup, then $\Prim_F(G)$ is trivial. We then analyse the
consequences of Theorem \ref{thm:quotients} for soluble groups, which leads,
in Theorem \ref{thm:onlyif}, to a proof of part \ref{item:structG} of Theorem
\ref{thm:main}.

To prove part \ref{item:generatorsPrim} of the theorem, it then remains to
exhibit explicit relations of the form $\Theta = G+\sum_{H\lneq G}a_HH$
for groups $G$ appearing in the theorem that are not $(p,q)$-Dress for any prime number $q$,
and to separately deal with $(p,q)$-Dress groups that do not have a non-trivial
normal $p$-subgroup, i.e. that are $q$-quasi-elementary. The main difference between
the case we are treating in this paper and the case of $F$ having characteristic $0$,
which was treated in \cite{bra1}, is that we do not have character theory at
our disposal. Instead, to
prove that an element of $\B(G)$ is a relation, we use Conlon's induction
theorem, Theorem \ref{thm:Conlon}, so we are led to computing fixed points
of various $G$-sets under all $p$-hypo-elementary subgroups of $G$.
Section \ref{sec:relations} is devoted to these somewhat technical calculations, and Proposition
\ref{prop:1dimrel} and Theorem \ref{thm:highdimrel} furnish the final ingredients
for the proof of Theorem \ref{thm:main}. The whole proof is summarised
at the end of Section \ref{sec:relations}.

We remark that for a full classification of Brauer relations in positive
characteristic, one would also need to determine the structure and generators of
$\Prim_F(G)$ for groups $G$ that are $(p,p)$-Dress groups. That problem is left
open in this work.

\begin{acknowledgements}
The first author was partially supported by a Research Fellowship from
the Royal Commission for the Exhibition of 1851, and by an EPSRC First
Grant, and the second author was supported by an EPSRC Doctoral Grant
during this project. Both authors were at the University of Warwick. We
would like to thank all these institutions for their financial support,
respectively for a conducive research environment. We would like to
also thank Jeremy Rickard for a helpful conversation, and an anonymous referee
for helpful remarks and corrections.
\end{acknowledgements}

\subsection*{Notation}
Throughout the rest of the paper, we fix a prime number $p$, and $F$ will denote
either a field of characteristic $p$, or a local ring with finite residue field
of characteristic $p$; $G$ will always denote a finite group; $O_p(G)$ is the
$p$-core of $G$, defined as the intersection of all its $p$-Sylow subgroups;
for a prime $q$, $O^q(G)$ will denote the minimal normal subgroup of $G$ of
$q$-power index; $\Aut(G)$ denotes the automorphism group of $G$, and
$\Out(G)$ denotes the outer automorphism group of $G$, i.e. the
quotient of $\Aut(G)$ by the subgroup of inner automorphisms.

If $H$, $U$ are two subgroups of $G$, and $g$, $x\in G$, then
we will write ${}^gx=gxg^{-1}$ and ${}^gU=gUg^{-1}$; the normaliser of
$H$ in $G$ will be denoted by $\N_G(H)$; the commutator $[H,U]$ is the subgroup of
$G$ generated by $\{[h,u]=huh^{-1}u^{-1}: h\in H, u\in U\}$.
The commutator $[H,U]$ is trivial if and only if every element of $H$ commutes
with every element of $U$.

If $H$ is a subgroup of $G$, then a (left) transversal for $G/H$ is a set
$T\subseteq G$ such that $G$ is a disjoint union $G=\bigsqcup_{g\in T}gH$.

The Frattini subgroup $\Phi(G)$ of $G$ is defined as the intersection of all
maximal subgroups of $G$. If $l$ is a prime, and $W$ is an $l$-group,
then the Frattini subgroup of $W$ is equal to $[W,W]W^l$. It has the property
that a normal subgroup $N$ of $W$ contains the Frattini subgroup if and only
if $W/N$ is an elementary abelian $l$-group. It also has the property that every
element of $\Phi(W)$ is a ``non-generator'', meaning that every generating set of $W$
remains a generating set if all elements of $\Phi(W)$ are omitted.

If $R$ is any set of prime numbers, then an $R$-Hall
subgroup of $G$ is a subgroup whose order is a product of primes in $R$, and
whose index is not divisible by any prime in $R$. Hall's theorem says that if
$G$ is soluble, then for every set $R$ of prime numbers, an $R$-Hall subgroup
of $G$ exists, any two $R$-Hall subgroups are conjugate, and every subgroup of
$G$ whose order is a product of prime numbers in $R$ is contained in some
$R$-Hall subgroup \cite[Theorems 3.13, 3.14, and Problem 3C.1]{Isaacs}. If $q$
is a prime number, we will say ``$(-q)$-Hall subgroup'', when we mean an $R$-Hall
subgroup for $R$ being the set of all prime numbers except for $q$.

\section{Basic properties and induction theorems}\label{sec:first}
Let $G$ be a finite group, let $H$ be a
subgroup of $G$, let $N$ be a normal subgroup of $G$, and let $\pi$ denote the
quotient map $G\to G/N$. There are maps

\beq
\Ind_{G/H}: & \K_F(H) & \rar & \K_F(G),\\
&\displaystyle \sum_{U\le H}n_UU & \mapsto & \displaystyle \sum_{U}n_UU;\\
\\
\Res_{G/H}: & \K_F(G) & \rar & \K_F(H),\\
& \displaystyle \sum_{U\le G}n_UU  & \mapsto & {\displaystyle \sum_{U\le G}\sum_{g\in H\backslash G/U} n_U({}^gU\cap H)};\\
\\
\Inf_{G/N}: & \K_F(\bar{G}) & \rar & \K_F(G),\\
            &\displaystyle \sum_{\bar{U}\le G/N}n_{\bar{U}}\bar{U} & \mapsto & \displaystyle \sum_{\bar{U}\le G/N}n_{\bar{U}}\pi^{-1}(\bar{U}),
\eeq
induced by the natural induction, restriction, and inflation
maps, respectively, on the Burnside rings.

\begin{lemma}\label{lem:closedunderquo}
Let $G$ be a finite group, let $N$ be a normal subgroup of $G$, and let
$q$ be a prime number. Then:
\begin{enumerate}[leftmargin=*, label={\upshape(\alph*)}]
\item\label{item:quasielemquo} if $G$ is $q$-quasi-elementary, then so is $G/N$;
\item\label{item:dressquo} if $G$ is a $(p,q)$-Dress group, then so is $G/N$.
\end{enumerate}
\end{lemma}
\begin{proof}
\ref{item:quasielemquo} Let $C$ be a normal cyclic subgroup of $G$ of $q$-power index.
Then $CN/N$ is a normal cyclic subgroup of $G/N$ of $q$-power index.

\ref{item:dressquo} The $p$-core of $G/N$ contains $O_p(G)N/N$, so $(G/N)/O_p(G/N)$ is a quotient
of $G/O_p(G)$. The assertion therefore follows from part (a).
\end{proof}

\begin{definition}\label{def:fixedpoints}
Given a $G$-set $X$ and a subgroup $U$ of $G$, define $f_U(X)$ to be
the number of fixed points in $X$ under the action of $U$. Extended linearly,
$f_U$ defines a ring homomorphism $\B(G)\to \Z$.
\end{definition}

Let $G$ be a finite group, let $\cC$ and $\cP$ denote full sets of
representatives of conjugacy classes of cyclic, respectively
$p$-hypo-elementary subgroups of $G$.

\begin{theorem}[Artin's Induction Theorem]\label{thm:Artin}
The ring homomorphism
$$
\prod_{U\in \cC}f_U\colon \B(G)\rar \prod_{U\in \cC}\Z
$$
has image of finite additive index, and its kernel is precisely equal
to $\K_\Q(G)$.
\end{theorem}
\begin{proof}
See the proof of \cite[Theorem 5.6.1]{Benson}.
\end{proof}

\begin{theorem}[Conlon's Induction Theorem]\label{thm:Conlon}
The ring homomorphism
$$
\prod_{U\in \cP}f_U\colon\B(G)\rar \prod_{U\in \cP}\Z
$$
has image of finite additive index, and its kernel is precisely equal to $\K_F(G)$.
\end{theorem}
\begin{proof}
See \cite[\S 81B]{CR2} or \cite[\S 5.5--5.6]{Benson}.
\end{proof}

\begin{corollary}\label{cor:ConlonRank}
The group $\K_F(G)$ is free abelian of rank equal to the number of conjugacy
classes of non-$p$-hypo-elementary subgroups of $G$.
\end{corollary}

\begin{corollary}\label{cor:Conlon}
There exists a Brauer relation over $F$ of the form $a_GG +\sum_{U\in \cP} a_UU\in \K_F(G)$,
where $a_G$, $a_U\in \Z$.
\end{corollary}
\begin{proof}
If $G$ is $p$-hypo-elementary, then the statement is empty, so assume
that $G$ is not $p$-hypo-elementary.
By Theorem \ref{thm:Conlon}, the set consisting of $F[G/G]$ and of
$F[G/U]$ for $U\in\cP$ is linearly dependent in $\R_F(G)$. On the other hand, it
is easy to see that if the elements $U_i$ of $\cP$ are ordered in non-descending order
with respect to size, then the matrix $(f_{U_i}(U_j))_{U_i,U_j\in \cP}$ is
triangular with non-zero entries on the diagonal, so by Theorem \ref{thm:Conlon}
the set $\{F[G/U]: U\in \cP\}$ is linearly independent in $\R_F(G)$. This proves
the corollary.
\end{proof}

Corollary \ref{cor:Conlon} is also often referred to as Conlon's Induction
Theorem.
%
%
The following theorem is the basic structure result on $\Prim_F$.
\begin{theorem}\label{thm:quotients}
Let G be a finite group that is not a $(p,q)$-Dress group for any prime number $q$.
Then the following trichotomy holds:
\begin{enumerate}[leftmargin=*, label={\upshape(\alph*)}]
\item\label{item:allquophypo} if all proper quotients of G are $p$-hypo-elementary, then $\Prim_F(G) \cong \Z$;
\item\label{item:allquodress} if there exists a prime number $q$ such that all proper
quotients of G are $(p,q)$-Dress groups, and at least one of them is not
$p$-hypo-elementary, then $\Prim_F(G) \cong \Z/q\Z$;
\item if there exists a proper quotient of $G$ that is not a $(p,q)$-Dress group
for any prime number $q$, or if there exist distinct prime numbers $q_1$ and $q_2$
and, for $i=1$ and $2$, a proper quotient of $G$ that is a non-$p$-hypo-elementary
$(p,q_i)$-Dress group, then $\Prim_{F}(G)$ is trivial.
\end{enumerate}
In cases \ref{item:allquophypo} and \ref{item:allquodress}, $\Prim_F(G)$ is
generated by any relation in which $G$ has coefficient $1$.
\end{theorem}
\begin{proof}
See \cite[Theorem 1.2]{GFIs}.
\end{proof}

\begin{corollary}\label{cor:SES}
Let $G$ be a finite group, and suppose that $\Prim_F(G)$ is non-trivial.
Then $G$ is an extension of the form
\beq
1\rightarrow S^d \rightarrow G \rightarrow D \rightarrow 1,
\eeq
where $S$ is a finite simple group, $d\ge0$ is an integer, and $D$ is a $(p,q)$-Dress group
for some prime number $q$.
Moreover, if $d\ge 1$ and $S$ is not cyclic, then the canonical map $D\rar \Out(S^d)$
is injective, and $S^d$ has no proper non-trivial subgroups that are
normal in $G$. In this case, $\Prim_F(G)\cong \Z$ if $D$ is $p$-hypo-elementary,
and $\Prim_F(G)\cong \Z/q\Z$ otherwise. 
\end{corollary}
\begin{proof}
If $G$ is a $(p,q)$-Dress group for some prime number $q$, then the assertion
is clear, so suppose that it is not.
The group $G$ has a chief series, so there exists a normal subgroup
$W\cong S^d$ of $G$, where $S$ is a simple group and $d\ge 1$.
By Theorem \ref{thm:quotients}, the quotient $G/W$ is a $(p,q)$-Dress
group for some prime number $q$.

Now suppose that $S$ is not cyclic. Let $K$ be the kernel of the map
$G\rar\Aut(S^d)$ given by conjugation. The centre of $S^d$ is trivial,
so $K \cap S^d =\{1\}$. If $K$ is non-trivial, then $G/K$ is a proper
quotient that is not soluble, and in particular not a $(p,q)$-Dress group,
contradicting Theorem \ref{thm:quotients}.
So $G$ injects into $\Aut(S^d)$, and thus
$G/S^d=D$ injects into $\Out(S^d)$. Similarly, if $N\normal G$
is a proper subgroup of $S^d$, then $G/N$ is not soluble, and in particular
not a $(p,q)$-Dress group, contradicting Theorem \ref{thm:quotients}.
Finally, the description of $\Prim_F(G)$ is given by Theorem \ref{thm:quotients}.
\end{proof}

\section{Main reduction in soluble groups}\label{sec:soluble}

In this section, we analyse $\Prim_F(G)$ for soluble groups $G$. The main results
of the section are Theorem \ref{thm:pqDress}, concerning $(p,q)$-Dress groups,
and Theorem \ref{thm:onlyif}, which gives necessary conditions on a soluble group
$G$ for $\Prim_F(G)$ to be non-trivial. The first of these is proved by comparing
the consequences of Conlon's Induction Theorem for $G$ and for its subquotients,
while the second is derived from a careful analysis of the
implications of Theorem \ref{thm:quotients} for soluble groups.

\begin{lemma}\label{lem:pqDress}
Let $q$ be a prime number different from $p$, and let
$G=P\rtimes(C\rtimes Q)$ be a $(p,q)$-Dress group, where $P$ is a $p$-group,
$Q$ is a $q$-group, and $C$ is a cyclic group of order coprime to $pq$.
Let $\cS$ be a full set of $G$-conjugacy class representatives of subgroups of
$P$. For each $U\in \cS$, let $N_U$ be a $(-p)$-Hall subgroup of $\N_G(U)$,
and let $\cT_U$ be a full set of $N_U$-conjugacy class representatives of subgroups
of $N_U$. Then:
\begin{enumerate}[leftmargin=*, label={\upshape(\alph*)}]
\item\label{item:conjugacy} for every $U\in \cS$, and all subgroups $V_1$, $V_2$
of $N_U$, $V_1$ and $V_2$ are $N_U$-conjugate if and only if they are $\N_G(U)$-conjugate;
\item\label{item:subgroups} for every subgroup $H$ of $G$, there exists a unique $U\in \cS$
and a unique $V\in \cT_U$ such that $H$ is $G$-conjugate to $U\rtimes V$.
\end{enumerate}
\end{lemma}
\begin{proof}
To prove part \ref{item:conjugacy}, let $U\in \cS$, and $V_1$, $V_2\le N_U$, and suppose that
there exists an element $g$ of $\N_G(U)$ such that ${}^gV_1 = V_2$.
Since $\N_G(U)=\N_P(U)\rtimes N_U$, we may write $g=nu$, where $u\in \N_P(U)$
and $n\in N_U$. Let $v\in V_1$. By assumption, ${}^gv\in V_2\subseteq N_U$, so
${}^uv\in N_U$, and hence $[v,u]=v(uv^{-1}u^{-1})\in N_U$. On the other hand,
$\N_P(U)=\N_G(U)\cap P$ is normal in $\N_G(U)$, so
$[v,u]=(vuv^{-1})u^{-1}\in \N_P(U)$. Since $\N_P(U)\cap N_U=\{1\}$, this
implies that $u$ and $v$ commute. Since $v$ was arbitrary, we deduce that
$u$ centralises $V_1$, so that ${}^gV_1={}^nV_1=V_2$, as claimed.

Now, we prove the existence statement of part \ref{item:subgroups}.
Let $H$ be a subgroup of $G$, and let $U=H\cap P$. After replacing $H$ with a
subgroup that is $G$-conjugate to it if necessary, we may assume that $U\in \cS$.
We then have $H\le \N_G(U)$. Let $V$ be a $(-p)$-Hall subgroup of $H$, which
is contained in a $(-p)$-Hall subgroup of $\N_G(U)$. Since all $(-p)$-Hall
subgroups of $\N_G(U)$ are conjugate to each other, we may assume, after
possibly replacing $H$ with a subgroup that is $\N_G(U)$-conjugate to it, that
$V$ is contained in $N_U$, so that after possibly replacing $H$ by a subgroup
that is $N_U$-conjugate to it, we may assume that $V\in \cT_U$,
which concludes the proof of the existence statement.

Finally, we prove uniqueness. Let $U_1$, $U_2\in \cS$, and let $V_i\in \cT_{U_i}$
for $i=1$, $2$ be such that $H_1=U_1\rtimes V_1$ is $G$-conjugate to
$H_2=U_2\rtimes V_2$. Since $U_i$ is the unique Sylow $p$-subgroup of $H_i$ for $i=1$ and $2$,
this implies that $U_1$ and $U_2$ are $G$-conjugate; and since both are contained
in $P$, and $\cS$ is assumed to be a complete set of distinct conjugacy class
representatives, this implies that $U_1=U_2$. Write $U=U_1$. We deduce that $H_1$ and
$H_2$ are $\N_G(U)$-conjugate. Since $V_i$ is a $(-p)$-Hall subgroup of $H_i$
for $i=1$ and $2$, it follows that $V_1$ and $V_2$ are also $\N_G(U)$-conjugate,
so by part \ref{item:conjugacy}, they are $N_U$-conjugate. Since $\cT_U$ is a
full set of representatives of $N_U$-conjugacy classes, we have $V_1=V_2$.
\end{proof}

\begin{thm}\label{thm:pqDress}
Let $q$ be a prime number different from $p$, and let $G$ be a $(p,q)$-Dress
group with non-trivial $p$-core. Then $\Prim_F(G)$ is trivial.
\end{thm}
\begin{proof}
We keep the notation of Lemma \ref{lem:pqDress}. In particular, we write
$G=P\rtimes(C\rtimes Q)$, where $P$ is a non-trivial $p$-group, $Q$ is a
$q$-group, and $C$ a cyclic group	of order coprime to $pq$.

For each $U\in \cS$, identify $N_U$ with $UN_U/U$ via the quotient map,
and consider the map
$$
\iota_U=\Ind_{G/UN_U}\Inf_{UN_U/U}\colon \B(N_U)\to \B(G).
$$
Let $I_U=\iota_U(\K_F(N_U))$. Note that all $\Theta\in I_U$ are imprimitive,
since either $U$ is non-trivial, so that $UN_U/U$ is a proper quotient,
or $N_U$ is a $(-p)$-Hall subgroup of $G$, which is proper since the $p$-core
of $G$ is assumed to be non-trivial.
We will now show that $\sum_{U\in \cS}I_U = \K_F(G)$.

First, we claim that each $\iota_U$ is injective. Inflation is always an injective
map of Burnside rings, so it suffices to show that the induction map $\Ind_{G/UN_U}$ is
injective on the image of $\Inf_{UN_U/U}$. Let $H_1$ and $H_2$ be subgroups of
$UN_U$ containing $U$ that are $G$-conjugate. Since the common $p$-core is $U$,
they are then $\N_G(U)$-conjugate. Since each of their respective
$(-p)$-Hall subgroups is contained in a $(-p)$-Hall subgroup of $UN_U$, and all
$(-p)$-Hall subgroups of $UN_U$ are conjugate, we may assume, replacing $H_1$ and
$H_2$ by $UN_U$-conjugate subgroups if necessary, that $H_i=UV_i$, where
$V_i\le N_U$ for $i=1$, $2$, and where $V_1$ is $\N_G(U)$-conjugate to $V_2$.
Lemma \ref{lem:pqDress}\ref{item:conjugacy} then implies that $V_1$ and $V_2$
are also $N_U$-conjugate, so $H_1$ and $H_2$ are $UN_U$-conjugate, and
injectivity of $\iota_U$ follows.

Next, we claim that the $I_U$ for $U\in \cS$ are linearly independent.
Indeed, suppose that $\sum_{U\in \cS}\Theta_U=0$, where $\Theta_U\in I_U$.
Let $U$ be maximal with respect to inclusion subject to the property that
$\Theta_U\neq 0$. Then all terms of $\Theta_U$ contain $U$, while for all elements
$U'\neq U$ of $\cS$, all terms of $\Theta_{U'}$ are contained in $U'N_{U'}$, which does not contain $U$.
Thus, for the sum to vanish, we must have $\Theta_U=0$ -- a contradiction.

A similar argument shows that $\sum_{U\in \cS}I_U$ is saturated in $\K_F(G)$:
suppose that $\sum_{U\in \cS}\Theta_U$ is divisible by some $n\in \Z_{\geq 2}$
in $\K_F(G)$ for $\Theta_U\in I_U$, and consider $U\in\cS$ that is maximal subject to the
property that $\Theta_U$ is not divisible by $n$ in $\K_F(U)$, or, equivalently,
in $\B(U)$; then note that the above argument shows that for every subgroup $H$
of $G$ that contains $U$, the coefficient of $H$ in $\Theta_{U'}$ is divisible
by $n$ for all elements $U'\neq U$ of $\cS$, so its coefficient in $\Theta_U$ must
also be divisible by $n$, so that in fact $\Theta_U$ is divisible by $n$ in $\B(G)$ -- a contradiction.

To prove equality, it therefore only remains to compare the ranks
of $\sum_{U\in \cS}I_U$ and of $\K_F(G)$. By linear
independence and by Corollary~\ref{cor:ConlonRank}, we have
\begin{align*}
\rank\left(\sum_{U\in \cS}I_U\right) =& \sum_{U\in \cS}\rank I_U\\
=&
\sum_{U\in \cS}\#\{\text{conjugacy classes of non-cyclic subgroups of $N_U$}\},
\end{align*}
and by Lemma \ref{lem:pqDress}\ref{item:subgroups} and Corollary \ref{cor:Conlon},
this is equal to the rank of $\K_F(G)$, which completes the proof.
\end{proof}
\begin{lemma}\label{lem:SESsplit}
Let $G$ be a finite group, and let $W$ be an abelian normal subgroup
with quotient $D$. Suppose that there exists a normal subgroup $H$ of $D$
such that $\gcd(\#H,\#W)=1$ and such that no non-identity element
of $W$ is fixed under the natural conjugation action of $H$ on $W$.
Then $G\cong W\rtimes D$.
\end{lemma} 
\begin{proof}
We may view $W$ as a module under $D$.
Since $H$ and $W$ have coprime orders, the cohomology group
$H^1(H,W)$ vanishes, so the Hochschild--Serre spectral sequence
gives an exact sequence
$$
H^2(D/H,W^H)\rightarrow H^2(D,W)\rightarrow H^2(H,W).
$$
The last term in this sequence also vanishes by the coprimality assumption,
while the first term vanishes, since $W^H$ is assumed to be trivial.
Hence $H^2(D,W)=0$, and therefore the extension $G$ of $D$ by $W$ splits.
\end{proof}

\begin{theorem}\label{thm:onlyif}
Let $G$ be a finite soluble group, and suppose that $\Prim_F(G)$ is
non-trivial. Then $G$ is one of the following:
\begin{enumerate}[leftmargin=*, label={\upshape(\roman*)}]
\item\label{item:quasielem} a quasi-elementary group $C\rtimes Q$ of order coprime
to $p$, where $C$ is cyclic and $Q$ is a $q$-group for some prime number $q$, and
either $C$ is not of prime order, or $Q$ does not act faithfully on $C$,
\item\label{item:WD} a semidirect product $G=W\rtimes D$, where $W=(C_l)^d$
for a prime number $l\neq p$ and an integer $d\ge 1$, and $D$ is a $(p,q)$-Dress
group for some prime number $q$, acting faithfully and irreducibly on $W$,
\item\label{item:2dim} $G=(C_l \rtimes D_1)\times (C_l \rtimes D_2)$, where
$l\neq p$ is a prime number, $D_1$, $D_2$ are cyclic $q$-groups for a prime number
$q$ that act faithfully on $C_l\times C_l$,
\item\label{item:ppDress} a $(p,p)$-Dress group.
\end{enumerate}
\end{theorem}
\begin{proof}
We begin by observing that if $G$ is a $(p,q)$-Dress group for some prime number
$q$, then the conclusion of the theorem holds. Indeed, if $G$ is a $(p,p)$-Dress
group, then this is clear. If, on the other hand, $G$ is a $(p,q)$-Dress group
for a prime number $q\neq p$, then
it follows from Theorem \ref{thm:pqDress} that $G$ must have trivial $p$-core,
so the order of $G$ is coprime to $p$,
which implies that all $p$-hypo-elementary subquotients of $G$ are cyclic. By
Theorems \ref{thm:Artin} and \ref{thm:Conlon}, we then have $\Prim_F(G)=\Prim_{\Q}(G)$,
and it follows from \cite[Theorem A, case (4)]{bra1} that $G$ satisfies
the conditions of part \ref{item:quasielem} of the theorem. In particular, if
$G$ is quasi-elementary, then the conclusion of the theorem holds.

We will repeatedly use this observation without further mention.

By Corollary \ref{cor:SES}, $G$ is an extension of the form
\begin{eqnarray}\label{eq:ses}
1\rightarrow W=(C_l)^d \rightarrow G \rightarrow D \rightarrow 1,
\end{eqnarray}
where $l$ is a prime number, $d\ge 0$ is an integer, and $D$ is a $(p,q)$-Dress
group for some prime number $q$.
If $d=0$ or $l=p$, then $G$ is a $(p,q)$-Dress group, and we are done. For the rest of
the proof, assume that $d\ge 1$ and $l\neq p$. We now consider several cases.

\textbf{Case 1: $l\nmid \#D$.} By the Schur-Zassenhaus theorem
\cite[Theorem 3.8]{Isaacs}, the short
exact sequence (\ref{eq:ses}) splits, so we have $G\cong W\rtimes D$,
and we may view $D$ as a subgroup of $G$.
Let $N\normal G$ be the centraliser of $W$ in $D$.

\textbf{Case 1(a): $N\neq \{1\}$ and $D$ is $p$-hypo-elementary.}
The subgroup $WN/N$ is normal in $G/N$. By Theorem \ref{thm:quotients},
$G/N$ is a $(p,q)$-Dress group for some prime number $q$. 
It follows that $D/N$ is also normal in $G/N$,
so $G/N= WN/N\times D/N$, so the commutator $[W,D]$
is contained in $N\le D$. But also, since $W$ is normal in $G$, this commutator
is contained in $W$, so it is trivial. It follows that $W$ commutes with $D$,
and $G$ is a $(p,l)$-Dress group.

\textbf{Case 1(b): $N\neq \{1\}$ and $D$ is not $p$-hypo-elementary.}
By Theorem \ref{thm:quotients}, $G/N$ is a $(p,q)$-Dress group. Since
$l\nmid pq$, this implies that $W$ must be cyclic, and, by
the same argument as in case 1(a), it must commute with $O^q(D)$.
It follows that $G$ is a $(p,q)$-Dress group.

\textbf{Case 1(c): $N=\{1\}$ and $D$ acts reducibly on $W$.}
Let $U$ be a proper non-trivial subgroup of $W$ that is normal in $G$. Since
$l\nmid \#D$, the $\F_l[D]$-module $W$ is semisimple, so there exists a subgroup
$V$ of $W$ that is normal in $G$ and such that $UV=W$ and $U\cap V=\{1\}$.
By Theorem \ref{thm:quotients}, both $G/U$ and $G/V$ are $(p,q)$-Dress groups.
Since $l\nmid pq$, this implies that $V\cong W/U\cong C_l$ and $U\cong W/V\cong C_l$.
Thus, $G\cong (U\rtimes D_1)\times (V\rtimes D_2)$, where $D_1$ acts faithfully
on $U$, and $D_2$ acts faithfully on $V$, and in particular both are cyclic.
It follows that $O_p(G/U)$ is of the form $NU/U$ for a $p$-subgroup $N$ of $D_1$.
For $G/U$ to be a $(p,q)$-Dress group, the $(-q)$-Hall subgroup of $G/UN$ must be
cyclic, which forces $D_2$ to be a $q$-group, and similarly for $D_1$.
This is case \ref{item:2dim} of the theorem.

\textbf{Case 1(d): $N=\{1\}$ and $D$ acts irreducibly on $W$.}
This is case \ref{item:WD} of the theorem.

\textbf{Case 2: $l\mid \#D$ and $G=W\rtimes D$.}
In this case, $N=\ker(D\rar \Aut W)$ is again a normal subgroup of $G$.

\textbf{Case 2(a): $N\neq \{1\}$.}
By Theorem \ref{thm:quotients}, the quotient $G/N$ is a $(p,q)$-Dress group.
Since $D/N$ acts faithfully on $W$, no non-trivial subgroup of $D/N$ can be
normal in $G/N$. In particular, $O_p(G/N)$ must be trivial, so $N$ contains
$O_p(D)$, and $G/N$ is in fact quasi-elementary, $G/N\cong C\rtimes Q$, where
$C$ is cyclic and $Q$ is a $q$-group. By the same argument, $C$ is an $l$-group.
Now, if $q=l$, then $G/N$ is an $l$-group, and $G$ is an extension of an $l$-group
by the $(p,l)$-Dress group $N$, hence is itself a $(p,l)$-Dress group.
If $q\neq l$, then $W$ must be cyclic, and must commute with $O_p(D)$, so $O_p(D)$
is normal in $G$, and $G/O_p(D)$ is $q$-quasi-elementary, whence $G$ is a
$(p,q)$-Dress group.

\textbf{Case 2(b): $N=\{1\}$ and $D$ acts reducibly on $W$.}
Let $U\le W$ be a non-zero
proper $\F_l[D]$-subrepresentation of $W$.
By Theorem \ref{thm:quotients}, the quotient $G/U$ is a $(p,q)$-Dress group.

\textbf{Case 2(b)(i): $l\neq q$.}
Then the $l$-Sylow subgroups of $G/U$ must be cyclic.
In particular, any $l$-Sylow subgroup $C$ of $D$, which is non-trivial by assumption,
acts trivially by conjugation on $W/U$. Since $G$ is assumed to be a semi-direct
product, the $l$-Sylow subgroup of $G/U$ is a direct product of $W/U$ and $C$,
and therefore cannot be cyclic -- a contradiction.

\textbf{Case 2(b)(ii): $l=q$.}
Either $G/U$ is an $l$-group, in which case so is
$G$, and we are in case \ref{item:quasielem} of the theorem; or there exists
a subgroup $C\le D$ of order coprime to $l$ such that $CU/U$ is normal in $G/U$,
and in particular $C$ is normal in $D$.
The $\F_l[C]$-module $W$ is then semisimple,
so there exists a subgroup $V\le W$ that is normalised by $C$, and
such that $VU=W$ and $V\cap U = \{1\}$. Since $CU/U$ is normal in $G/U$, and
since $W/U$ is also normal in $G/U$, $CU/U$ and $W/U$ commute, so we
have $[C,V]\le U$. But since $V$ is normalised by $C$, we also have $[C,V]\le V$,
so $C$ in fact centralises $V$. Thus, $V$ is contained in $W^C$, which is
a normal subgroup of $G$. If $W^C=W$, then $C\le N$, contradicting
the assumption that $N=\{1\}$. So $W^C$ is a proper non-trivial subgroup
of $W$.
Since $l\nmid \#C$, there exists a non-trivial
subgroup $U'\le W$ such that $W=U'W^C$ and $U'\cap W^C = \{1\}$. In particular,
$(U')^C=\{1\}$. By Theorem \ref{thm:quotients}, the quotient $G/W^C$ is
$(p,l)$-Dress, so $CW^C/W^C$ is contained in the normal subgroup
$O^l(G/W^C)=O^l(D)W^C/W^C$. It follows that $[C,U']\le W^CO^l(D)$. But since
$U'$ is normalised by $C$, we also have $[C,U']\le U'$. Since
$U'\cap W^CO^l(D)$ is trivial, we deduce that $C$ centralises $U'$ -- a contradiction.

\textbf{Case 2(c): $N=\{1\}$, and $D$ acts irreducibly on $W$.}
This is case \ref{item:WD} of the theorem.

\textbf{Case 3: $l\mid \#D$ and the extension of $D$ by $W$ is not split.}
By the Schur-Zassenhaus Theorem, the preimage of $O_p(D)$ under the quotient map $G\to G/W$
is a split extension by $W$. Let $P$ be a complement
to $W$ in this preimage. In other words, $P$ is a subgroup of $G$
that maps isomorphically onto $O_p(D)$ under the quotient map $G\rar G/W$.

\textbf{Case 3(a): $P=\{1\}$ and $l\neq q$.}
Then the $l$-Sylow subgroup $S$ of $G$ is normal in $G$. If it is elementary
abelian, then the extension of $D$ by $S$ splits by the Schur-Zassenhaus theorem
\cite[Theorem 3.8]{Isaacs},
and we are in Case 2 of the proof. Otherwise, the Frattini subgroup
$\Phi=[S,S]S^l$ of $S$ is non-trivial, and since it is a characteristic subgroup
of $S$, it is normal in $G$. By Theorem \ref{thm:quotients}, the quotient
$G/\Phi$ is a $(p,q)$-Dress group, so the $l$-Sylow subgroup of $G/\Phi$ is
cyclic. But since $\Phi$ consists of ``non-generators'' of $S$, this implies
that $S$ itself is cyclic, so $G$ is $q$-quasi-elementary.

\textbf{Case 3(b): $P=\{1\}$ and $p\neq l=q$.}
Let $C$ be a $(-l)$-Hall subgroup of $G$. The assumptions on $G$ imply
that $C$ is cyclic, and that $D$ is of the form $C\rtimes Q$, where $Q$
is a $q$-group. If $W^C=W$, then $C$ is a normal subgroup of $G$,
and $G$ is $q$-quasi-elementary. If $W^C=\{1\}$, then Lemma \ref{lem:SESsplit}
implies that the extension of $D$ by $W$ splits -- a contradiction.
So $W^C$ is a non-trivial proper subgroup of $W$, which is normal in $G$, 
since $C$ is normal in $D$. Since the order of $C$ is coprime to $l$,
the $\F_l[C]$-representation $W$ is semisimple, so there exists a subgroup
$U$ of $W$ that is normalised by $C$, and such that $UW^C=W$, $U\cap W^C=\{1\}$.
By Theorem \ref{thm:quotients}, the quotient $G/W^C$ is a $(p,q)$-Dress group.
But it has trivial $p$-Sylow subgroup, so it is $q$-quasi-elementary,
and $CW^C/W^C$ is normal in $G/W^C$. Thus $[C,U]\le W^C$. But also,
$U$ is a $C$-subrepresentation, so $[C,U]\le U$, whence we deduce that $C$
centralises $U$, so that $W^C=W$ -- a contradiction.

\textbf{Case 3(c): $P\neq \{1\}$ and $W^P=W$.}
In this case, $P$ is a non-trivial normal $p$-subgroup of $G$. By Theorem \ref{thm:quotients},
the quotient $G/P$ is $(p,q)$-Dress, therefore so is $G$ itself.

\textbf{Case 3(d): $P\neq \{1\}$ and $W^P\neq W$.}
By Lemma \ref{lem:SESsplit}, the subgroup $W^P$ is non-trivial.
Moreover, since $P$ is a normal subgroup of $D$, $W^P$ is a normal subgroup of $G$.
The $\F_l$-representation $W$ of $P$ is semisimple, so there exists a subgroup
$U\le W$ that is normalised by $P$ and such that $UW^P=W$, $U\cap W^P=\{1\}$.
By Theorem \ref{thm:quotients}, the quotient $G/W^P$ is a $(p,q)$-Dress group.
We claim that $O_p(G/W^P)$ must be trivial. Indeed, $O_p(G/W^P)$ is necessarily
of the form $NW^P/W^P$, where $N$ is a subgroup of $P$ that is normal in $D$.
But then we have $[N,U]\le W^P$, and also $[N,U]\le U$, since $U$ is a
$P$-subrepresentation of $W$. Thus $N$ centralises $U$, whence $W^N=W$.
By Lemma \ref{lem:SESsplit}, the assumption that the extension of $D$ by $W$ is
non-split forces $N=\{1\}$.

\textbf{Case 3(d)(i): $l\neq q$.}
Then the $l$-Sylow subgroup of $G/W^P$ must be cyclic and normal in $G/W^P$.
Since $W^P\neq W$, and since we assume that $l\mid \#D$, this implies that the $l$-Sylow
subgroup $S$ of $G$ is normal in $G$ and has an element of order strictly greater than $l$.
Thus, the Frattini subgroup $\Phi=[S,S]S^l$ of $S$ is non-trivial, and since it
is a characteristic subgroup of $S$, it is normal in $G$. By Theorem
\ref{thm:quotients}, the quotient $G/\Phi$ is a $(p,q)$-Dress group, so the $l$-Sylow
subgroup of $G/\Phi$ is cyclic. But that implies that the $l$-Sylow subgroup
of $G$ is also cyclic, and therefore $W\cong C_l$,
contradicting the assumptions that $\{1\}\neq W^P\neq W$.

\textbf{Case 3(d)(ii): $l=q$.}
Then $p\neq q$, so the $p$-Sylow subgroup of the $(p,q)$-Dress group
$G/W^P$ must be normal in $G/W^P$, contradicting the observation that $O_p(G/W^P)$
is trivial.

This covers all possible cases, and concludes the proof.
\end{proof}

\section{Explicit relations}\label{sec:relations}
In the present section, we prove Theorem \ref{thm:main}.
Proposition \ref{prop:1dimrel} below proves parts
\ref{item:generatorsPrim}\ref{item:main_Serre}\ref{item:mn} and
\ref{item:generatorsPrim}\ref{item:main_Serre}\ref{item:q^k}
of the theorem. The main remaining step is to prove that the element appearing
in part
\ref{item:generatorsPrim}\ref{item:main_Serre}\ref{item:d>1} of the Theorem is indeed an element of $\K_F(G)$, and that
is achieved in Theorem \ref{thm:highdimrel}. Most of the section is devoted to
the proof of Theorem \ref{thm:highdimrel}. With all the ingredients in place,
the proof of Theorem \ref{thm:main} is assembled from them at the end of the section.

\begin{proposition}\label{prop:1dimrel}
Let $l\neq p$ be a prime number, and let
$G=C_l \rtimes C$, where $C$ is a non-trivial cyclic group,
acting faithfully on $C_l$.
Then $\Prim G\cong \Z$, and is generated by the following relation $\Theta$:
\begin{enumerate}[leftmargin=*, label={\upshape(\roman*)}]
\item if $C\cong C_{mn}$, where $m$, $n>1$ are coprime integers,
then
$$
\Theta=G-C+\alpha(C_n-C_l\rtimes C_n)+\beta(C_m-C_l\rtimes C_m),
$$
where $\alpha$, $\beta$ are any integers satisfying $\alpha m+\beta n = 1$;
\item if $C\cong C_{q^{k+1}}$, where $q$ is a prime number, and $k\in\Z_{\ge 0}$, then
$$
\Theta=C_{q^k}-qC-C_l\rtimes C_{q^k}+qG.
$$
\end{enumerate}
\end{proposition}
\begin{proof}
The hypotheses on $G$ imply that all non-cyclic subquotients of $G$ have trivial
$p$-core, so a subquotient of $G$ is cyclic if and only if it is $p$-hypo-elementary.
It therefore follows from Theorems \ref{thm:Artin} and \ref{thm:Conlon}, that
$\B_F(G)=\B_\Q(G)$, and $\Prim_F(G)=\Prim_\Q(G)$, and the result follows
from \cite[Theorem A, case 3a]{bra1}.
\end{proof}

\begin{theorem}\label{thm:highdimrel}
Let $l\neq p$ and $q$ be prime numbers, let $G=W\rtimes Q$, where
$W=(C_l)^d$ with $d\in \Z_{\ge 2}$, and $Q$ is a $(p,q)$-Dress
group acting faithfully on $W$. Suppose that either
$Q$ acts irreducibly on $W$, or
$d=2$, and $G=(C_l\rtimes P_1)\times (C_l\rtimes P_2)$, where the
$P_i$ are $q$-groups acting faithfully on the respective factor of $W$.
Then the element
$$
\Theta = G-Q+\sum_{U\le_G W\atop (W:U)=l} (U\N_Q(U)-W\N_Q(U)),
$$
of $\B(G)$ is in $\K_F(G)$, where the sum runs over a full set of $G$-conjugacy
class representatives of index $l$ subgroups of $W$.
\end{theorem}

The proof of the theorem will require some preparation.

Recall from Definition \ref{def:fixedpoints} that if $X$ is a $G$-set, and
$U$ is a subgroup of $G$, then $f_U(X)$ denotes the number of fixed points
in $X$ under $U$, and that this extends linearly to a ring homomorphism
$f_U\colon \B(G)\to \Z$.

\begin{lemma}\label{lem:fixedpointcount}
Let $G$ be a finite group, and let $H$ and $K$ be subgroups. Then
$f_K(H)=\#\{g\in G/H : K\subseteq {}^gH\}=\#\{g\in G/H : {}^gK\subseteq H\}$.
\end{lemma}
\begin{proof}
By Mackey's formula for $G$-sets, we have
$$
\Res_K(G/H)=\bigsqcup_{g\in K\backslash G/H}K/({}^gH\cap K).
$$
By definition, $f_K(H)$ is the number of singleton orbits under the action of
$K$ on $G/H$, so $f_K(H)=\#\{g\in K\backslash G/H : K\subseteq {}^gH\}$.
An explicit calculation shows that the map $G/H\to K\backslash G/H$, $gH\mapsto KgH$
defines a bijection between $\{g\in G/H : K\subseteq {}^gH\}$ and
$\{g\in K\backslash G/H : K\subseteq {}^gH\}$, which proves the first equality.
The second equality is clear.
\end{proof}

\begin{lemma}\label{lem:Soc_eq_head}
Let $G$ be a finite group, let $l$ be a prime number, and let $\ck$
be a field of characteristic $l$. Suppose that there exists a normal subgroup
$N$ of $G$ such that $l\nmid\#N$ and $G/N$ is a cyclic $l$-group. Then for
every $\ck[G]$-module $M$, we have $\dim_{\ck} M^G = \dim_{\ck} M_G$. Moreover,
if $M$ is an indecomposable $\ck[G]$-module, then this dimension is $0$ or $1$.
\end{lemma}
\begin{proof}
Let $M$ be a $\ck[G]$-module. We may, without loss of generality, assume that $M$
is indecomposable. The element $e=(1/\#N)\sum_{n\in N}n\in \ck[G]$ is a central idempotent,
and we have $M^N=eM$. If $M^G=0$, then it follows from the assumption that $G/N$ is
an $l$-group that $M^N=0$ also. Since $l \nmid \#N$, the $N$-module $M$ is
semisimple, so $M_N=0$ also, so a fortioti $M_G=0$, and we are done.

Suppose that $M^G\neq 0$, so $eM\ne 0$. Since $M=eM \oplus (1-e)M$, and $M$ is indecomposable,
it follows that $eM=M$, so that $M$ is an indecomposable $\ck[G/N]$-module. Since
$G/N$ is a cyclic $l$-group, it follows from \cite{Janusz,Kupisch} that the 
maximal semisimple submodule and the maximal semisimple quotient module of
$M$ are both simple. But the only simple $\ck[G/N]$-module is the trivial
one, which completes the proof.
\end{proof}


\begin{lemma}\label{lem:KinQ}
Let $l$ be a prime number, let $d\ge 1$ be an integer, let $G=W\rtimes Q$,
where $W=(C_l)^d$ and $Q$ is any subgroup of $G$. Let $\Theta$
be the element of $\B(G)$ given by
$$
\Theta = G-Q+\sum_{U\le_G W\atop (W:U)=l} (U\N_Q(U)-W\N_Q(U)),
$$
where the sum runs over a full set of $G$-conjugacy class representatives
of index $l$ subgroups of $W$. Then for every subgroup $K$ of $Q$, we have
$f_K(\Theta)=\#(W_K)-\#(W^K)$.
\end{lemma}
\begin{proof}
For $w\in W$, we have $K\le {}^wQ$ if
and only if $(w^{-1}kwk^{-1})k\in Q$ for all $k\in K$. Since the bracketed
term is in $W$, this is equivalent to $w^{-1}kwk^{-1}=1$ for all $k\in K$,
i.e. to $w\in W^K$. Since $W$ forms a transversal for $G/Q$, it follows from
Lemma \ref{lem:fixedpointcount} that
\begin{align}
f_K(G) & = 1,\label{eq:fixedpointsG}\\
f_K(Q) & = \#\{w\in W :K\le {}^wQ\}=\#(W^K).\label{eq:fixedpointsQ}
\end{align}
We now calculate the remaining terms in $f_K(\Theta)$. Let $U\le W$ be a
subgroup of index $l$. Let $T\subseteq Q$ be a transversal for $G/W\N_Q(U)$.
Let $x\in W\setminus U$. Then a transversal for $G/U\N_Q(U)$ is given by
$\{tx^m:t\in T, 0\le m\le l-1\}$. Applying Lemma \ref{lem:fixedpointcount}, and noting
that $K\le {}^tQ$ for all $t\in T$, we have
$$
f_K(W\N_Q(U))=\#\{t\in T : K\le {}^t\N_Q(U)\},
$$ and
\begin{eqnarray*}
\lefteqn{f_K(U\N_Q(U))=}\\
& & \#\{(t,m)\in T\times \{0,\ldots,l-1\} : K\le {}^{tx^m}(U\N_Q(U))\}.
\end{eqnarray*}
To count that last number, we note that for all $k\in K$, and for all $y=tx^m$ in
the above transversal,
we have ${}^{y^{-1}}k = (x^{-m}t^{-1}ktx^{m}t^{-1}k^{-1}t)(t^{-1}kt)$,
and of the two bracketed terms the first is in $W$, and is equal to
$[x^{-m},t^{-1}kt]$, while the second is in $Q$. It follows that we have
$K\le {}^y(U\N_Q(U))$ if and only if $[x^{-m},t^{-1}Kt]\le U$ and $t^{-1}Kt\le \N_Q(U)$.
If $m\neq 0$, then these conditions are equivalent to
$[\langle x\rangle,t^{-1}Kt]\le U$ and $t^{-1}Kt\le \N_Q(U)$, and in particular are
independent of $m$.
Partitioning the transversal $\{tx^m:t\in T, 0\le m\le l-1\}=T\sqcup
\{tx^m:t\in T, 1\le m\le l-1\}$, we find that
\begin{eqnarray*}
\lefteqn{f_K(U\N_Q(U)-W\N_Q(U))}\\
& & = (l-1)\cdot\#\{t\in T: t^{-1}Kt\le \N_Q(U), [\langle x\rangle,t^{-1}Kt]\le U\}\\
& & = (l-1)\cdot\#\{t\in T: K\le \N_Q({}^tU), K\text{ acts trivially on }W/{}^tU\}.
\end{eqnarray*}
As $t$ runs over $T$, ${}^tU$ runs once over the $G$-orbit of $U$,
since $T$ is a transversal for $G/W\N_Q(U)=G/\N_G(U)$. It follows that if we take the sum
of the above expression over a full set of representatives $U$ of $G$-conjugacy
classes of index $l$ subgroups of $W$, we obtain
\begin{eqnarray}
\lefteqn{\sum_{U\le_G W\atop (W:U)=l}f_K(U\N_Q(U)-W\N_Q(U))}\nonumber\\
& & = (l-1)\#\{\text{quotients of }W_K\text{ of order $l$}\} =\#(W_K)-1.\label{eq:sumU}
\end{eqnarray}
The result follows by combining equations (\ref{eq:fixedpointsG}),
(\ref{eq:fixedpointsQ}), and (\ref{eq:sumU}).
\end{proof}

\begin{lemma}\label{lem:hyperplane}
Let $G=W\rtimes Q$ be a soluble group, where $W=(C_l)^d$ for some prime number
$l\ne p$, so that $W$ is naturally an $\F_l[G]$-module, and let $K\le G$ be a
subgroup of the form $K=K_{l'}\rtimes \langle\gamma\rangle$, where $K_{l'}$ is
contained in $Q$ and is of order coprime to $l$, and $\gamma=wh$ is of order a
power of $l$, with $w\in W$ and $h\in Q$. Suppose that $K$ is not $G$-conjugate
to any subgroup of $Q$. Then there exists an $\F_l[K]$-submodule $U$ of $W$ of
index $l$ not containing $w$. Moreover, for any such
$U\le W$, the group $K$ acts trivially on $W/U$.
\end{lemma}
\begin{proof}
First, we claim that $w\in W^{K_{l'}}$. Let $k\in K_{l'}$ be arbitrary. Since
$K_{l'}$ is normal in $K$, we have $whkh^{-1}w^{-1}\in K_{l'}\subset Q$.
But also, since $whkh^{-1}w^{-1} = [w, hkh^{-1}]hkh^{-1}$, and $hkh^{-1}\in Q$,
it follows that $[w, hkh^{-1}]\in Q$. On the other hand, since $w\in W$, and $W$
is normal in $G$, we also have $[w, hkh^{-1}]\in W$, hence $[w, hkh^{-1}]=1$, or equivalently
$kh^{-1}w^{-1}hk^{-1} =h^{-1}w^{-1}h$. We deduce that $h^{-1}wh\in W^{K_{l'}}$.
But since $K_{l'}$ is normal in $K$, the subgroup $W^{K_{l'}}$ is a $K$-submodule of $W$,
so that $w={}^\gamma(h^{-1}wh)\in W^{K_{l'}}$ also, as claimed.



Let $W^{K_{l'}}=N\oplus N'$ as a $K$-module, where $N$ is an indecomposable $K$-module
containing $w$. Since $K_{l'}$ acts trivially on $N$, we may view it as an
$\mathbb{F}_l[\langle \gamma \rangle ]$-module. Let $e_1$, \ldots, $e_k$ be an
$\F_l$-basis of $N$ with respect to which $\gamma$ acts in Jordan normal form. Then we
claim that $w$ is not contained in the proper $K$-submodule $L$ generated by
$e_1$, \ldots, $e_{k-1}$. Indeed, if it were, say
$w=e_1^{\alpha_1}\cdots e_{k-1}^{\alpha_{k-1}}$ for $\alpha_1$, \ldots, $\alpha_{k-1}\in \Z$,
then the element $e_2^{-\alpha_1}\cdots e_k^{-\alpha_{k-1}}$ would conjugate $wh$
to $h$ and would commute with $K_{l'}$, thus conjugating $K$ to a subgroup of $Q$,
which contradicts the hypotheses on $K$. Thus, the submodule $U=L\oplus N'$ satisfies
the conclusions of the lemma.

Finally, for any $U$ satisfying those conclusions,
$K_{l'}$ acts trivially on $W/U$, since it centralises $w\not\in U$. Moreover, $K/K_{l'}$ is
an $l$-group, so also acts trivially on $W/U$, since that quotient has order $l$.
\end{proof}

\begin{lemma}\label{lem:lminus1to1}
Let $G=W\rtimes Q$, $K=K_{l'}\rtimes\langle \gamma\rangle$, $w\in W$, and
$U\leq W$ be as in Lemma \ref{lem:hyperplane}. Let $\cS_1$ be the set of subgroups
of $W$ of index $l$ that are normalised by $K$, do not contain $w$, and are
different from $U$, and let
$\cS_2$ be the set of subgroups $V$ of $W$ of index $l$ that are normalised by
$K$, contain $w$, and such that $K$ acts trivially on $W/V$ by conjugation.
Then $\#\cS_1=(l-1)\cdot\#\cS_2$.
\end{lemma}
\begin{proof}
Suppose that either of $\cS_1$, $\cS_2$ is non-empty, let $U'\in \cS_1\cup \cS_2$.
Then $U\cap U'$ is a $K$-submodule of $W$ of index $l^2$, and the $\F_l[K]$-module
$W/(U\cap U')$ has at least two distinct quotients of order $l$ with trivial
$K$-action, namely $W/U$ and $W/U'$. It follows that the $\F_l[K]$-module
$W/(U\cap U')$ splits completely as a direct sum of two trivial $\F_l[K]$-modules.
Thus, there exist exactly $l+1$ index $l$ submodules of $W$ containing $U\cap U'$,
one of them equal to $U$, exactly one of them containing $w$, and thus in $\cS_2$,
and $l-1$ distinct elements of $\cS_1$. This proves that the map
$\cS_1\to \cS_2$, $U''\mapsto \langle w\rangle (U\cap U'')$ is $(l-1)$ to $1$,
and hence the lemma.
\end{proof}

\begin{lemma}\label{lem:KinUNU}
Let $G=W\rtimes Q$ be as in Lemma \ref{lem:hyperplane}. Let $K=K_{l'}\rtimes \langle \gamma\rangle$
be a subgroup of $G$, where $K_{l'}$ is contained in $Q$ and has order coprime to $l$, and $\gamma$
has order a power of $l$. Let $U\leq W$ be a subgroup of index $l$,
let $t\in Q$, let $x\in W\setminus U$, and let $m\in \{1,\ldots,l-1\}$.
Then the following are equivalent:
\begin{enumerate}[leftmargin=*, label={\upshape(\roman*)}]
\item\label{item:allm} for all $n\in \{1,\ldots,l-1\}$ we have $K \le {}^{tx^n}(U\N_Q(U))$;
\item\label{item:existsm} we have $K \le {}^{tx^m}(U\N_Q(U))$;
\item\label{item:3} we have $K\le {}^t\N_G(U)$, $[{}^t\langle x \rangle , K]\le {}^tU$, and $w\in {}^tU$.
\end{enumerate}
\end{lemma}
\begin{proof}
We will first show that \ref{item:existsm} is equivalent to \ref{item:3}.
We clearly have $K \le {}^{tx^m}(U\N_Q(U))$ if any only if
\begin{enumerate}[leftmargin=*, label={\upshape(\alph*)}]
\item $K_{l'} \le {}^{tx^m}(U\N_Q(U))$ and
\item $\gamma \in {}^{tx^m}(U\N_Q(U))$.
\end{enumerate}
First, we discuss (a). Let $k\in K_{l'}$.
Then
$$
x^{-m}t^{-1}ktx^m=x^{-m}(t^{-1}kt)x^{m}(t^{-1}kt)^{-1}(t^{-1}kt),
$$ where the last bracketed
term is in $Q$, and the expression preceding it is in $W$ and equals $[x^{-m},t^{-1}kt]$.
It follows that (a) is equivalent to $[{}^t(x^{-m}),K_{l'}]\le {}^tU$
and $K_{l'}\le {}^t\N_Q(U)$. Moreover, since $(t^{-1}kt)x^{m}(t^{-1}kt)^{-1}\in W$, and $W$
is abelian, we have $x^{-m}(t^{-1}kt)x^{m}(t^{-1}kt)^{-1}=(x(t^{-1}kt)x^{-1}(t^{-1}kt)^{-1})^{-m}$,
so that $[{}^t(x^m),K_{l'}]\le {}^tU$ if and only if $[{}^t\langle x\rangle,K_{l'}]\le {}^tU$.
In summary, (a) is equivalent to $[{}^t\langle x\rangle,K_{l'}]\le {}^tU$ and
$K_{l'}\le {}^t\N_Q(U)$.

We analyse condition (b) similarly. Write $\gamma=wh$, where $w\in W$ and $h\in Q$.
Then by the same calculation as before, (b) is equivalent to
$[{}^t\langle x\rangle,\gamma]w\le {}^tU$ and $h\in {}^t\N_Q(U)$. But if
$h\in {}^t\N_Q(U)=\N_Q({}^tU)$, then $\gamma$ normalises ${}^tU$ and, having order a power of
$l$, acts trivially on the quotient $W/{}^tU$, so that in this case
$\gamma ({}^tx)^{-1}\gamma^{-1}=({}^tx)^{-1}u$ for some $u\in {}^tU$. We then have
$[{}^tx,\gamma]w = uw$, and the condition that this is in ${}^tU$ is equivalent
to $w\in {}^tU$, so condition (b) is equivalent to
$[{}^t\langle x\rangle,\gamma]\le {}^tU$ and $w\in {}^tU$.
This proves the equivalence between \ref{item:existsm} and \ref{item:3}.

Since the condition \ref{item:3} does not depend on $m$, this also proves
the equivalence between \ref{item:allm} and \ref{item:existsm}.
\end{proof}

We are now ready to prove Theorem \ref{thm:highdimrel}.
\begin{proof}[Proof of Theorem \ref{thm:highdimrel}]
By Theorem \ref{thm:Conlon}, the statement of the theorem is equivalent to the claim
that for all $p$-hypo-elementary subgroups $K$ of $G$, we have $f_K(\Theta)=0$.

If $K$ is a $p$-hypo-elementary subgroup of $G$, then either
$K$ is $G$-conjugate to a subgroup of $Q$; or Hall's Theorem implies that
the $(-l)$-Hall subgroup of $K$, which is necessarily normal in $K$, is conjugate
to a subgroup of $Q$, so that, possibly after replacing with a conjugate subgroup,
$K$ is as in Lemmas \ref{lem:hyperplane} and \ref{lem:lminus1to1}. 

If $K$ is a $p$-hypo-elementary subgroup that is conjugate to a subgroup of $Q$,
then by Lemma \ref{lem:KinQ}, we
have $f_K(\theta)=\#W_K-\#W^K$, which is equal to $0$ by Lemma \ref{lem:Soc_eq_head}.

Suppose that $K=K_{l'}\rtimes\langle \gamma\rangle$, where $K_{l'}$ is
of order comprime to $l$ and is contained in $Q$, and $\gamma$ has order a power
of $l$, and assume that $K$ is not conjugate to a subgroup of $Q$. Then we have
$f_K(G)=1$ and $f_K(Q)=0$. Write
$\gamma=wh$, where $w\in W$ and $h\in Q$. Let $U\le W$ have index $l$, let
$T\subseteq Q$ be a transversal for $G/W\N_Q(U)$, and let $x\in W\setminus U$,
so that a transversal for $G/U\N_Q(U)$ is given by $\{tx^m:t\in T, 0\leq m\leq l-1\}$.
Then by Lemma \ref{lem:fixedpointcount}, we have
\begin{align*}
	f_K(W\N_Q(U))&=\#\{t\in T : K\le {}^t(W\N_Q(U))\},
	\\
	f_K(U\N_Q(U))&=\#\{t\in T : K\le {}^t(U\N_Q(U))\} +\\
	&\#\{(t,m)\in T\times \{1,\ldots,l-1\}: K\le {}^{tx^m}(U\N_Q(U))\}.
	\end{align*}
For $t\in T$, the condition that $K\subseteq {}^t(W\N_Q(U))$ and
$K\nsubseteq {}^t(U\N_Q(U))$ is equivalent to $K\le {}^t\N_G(U)$ and
$w\notin {}^tU$. Combining these observations with Lemma \ref{lem:KinUNU}, we have
\begin{eqnarray*}
\lefteqn{f_K(U\N_Q(U)-W\N_Q(U))}\\
 & = &\#\{(t,m)\in T\times \{1,\ldots,l-1\}: K\le {}^{tx^m}(U\N_Q(U))\}-\\
 & &\#\{t\in T : K\le {}^t\N_G(U), w\notin {}^tU\}) \\
& = &(l-1)\#\{t\in T: K\le {}^t\N_G(U), [{}^t\langle x \rangle , K]\le {}^tU, w\in {}^tU\}-\\
& & \#\{t\in T : K\le {}^t\N_G(U), w\notin {}^tU\}).
\end{eqnarray*}
For $K\le {}^t\N_G(U)=\N_G({}^tU)$, the condition $[{}^t\langle x \rangle , K]\le {}^tU$ is equivalent
to the condition that $K$ acts trivially on the quotient $W/{}^tU$. Since $T$
is a transversal for $G/\N_G(U)$,
it follows that as $t$ runs over $T$, ${}^tU$ runs exactly once over the
$G$-orbit of $U$. Hence, summing over
representatives of $G$-orbits of hyperplanes of $W$, we deduce
\begin{eqnarray*}
\lefteqn{f_K(\Theta) = 1 + }\\
& & (l-1)\#\{U\le W: (W:U) = l, K\le \N_G(U), w\in U, (W/U)^K=W/U\}-\\
& & \#\{U\le W: (W:U)=l, K\le\N_G(U), w\not\in U\}.
\end{eqnarray*}
By Lemma \ref{lem:lminus1to1}, this is equal to $0$.
\end{proof}

\begin{proof}[Proof of Theorem \ref{thm:main}]
Part \ref{item:structG} follows from Corollary \ref{cor:SES} if $G$ is not
soluble, and from Theorem \ref{thm:onlyif} if $G$ is soluble.
Part \ref{item:generatorsPrim}\ref{item:main_qquasielem} follows by combining
Theorems \ref{thm:Artin} and \ref{thm:Conlon}. Suppose that $G$ is as in part
\ref{item:structG}\ref{item:main_Serre}. If either $d>1$ or $D$ is not of prime power order,
then $G$ is not a $(p,q')$-Dress group for any prime number $q'$, while it is
easy to see that all its proper quotients are $(p,q)$-Dress groups, so
by Theorem \ref{thm:quotients} $\Prim_F(G)$ has the claimed structure,
and is generated by any relation in which $G$ has coeffiecient $1$. Thus,
part \ref{item:generatorsPrim}\ref{item:main_Serre}\ref{item:mn} follows from Theorem
\ref{prop:1dimrel}(i), while part \ref{item:generatorsPrim}\ref{item:main_Serre}\ref{item:d>1}
follows from Theorem \ref{thm:highdimrel}. The quasi-elementary case, part
\ref{item:generatorsPrim}\ref{item:main_Serre}\ref{item:q^k} follows from Proposition
\ref{prop:1dimrel}(ii). Finally, part \ref{item:generatorsPrim}\ref{item:main_nonabelianSerre}
follows from Corollary \ref{cor:SES}.
\end{proof}

\end{document}